\documentclass[a4paper, draf, 10pt]{amsart}
\usepackage{amssymb,amsfonts, amsmath, amsthm}
\usepackage[active]{srcltx}
\usepackage{datetime}\usdate

\newcommand{\supast}{{${^\ast}$}}

\newcommand{\D}{{\mathbb{D}}}

\newcommand{\C}{{\mathbb{C}}}

\newcommand{\Z}{{\mathbb{Z}}}

\newcommand{\alg}{{\mathcal A}}
\newcommand{\loneZ}{{\ell^1(\Z)}}
\newcommand{\linfZ}{{\ell^\infty(\Z)}}
\newcommand{\HerA}{{\alg_S}}

\newcommand{\st}{{\omega}}

\theoremstyle{plain}

\newtheorem*{theorem*}{Theorem}
\newtheorem*{lemma*}{Lemma}
\newtheorem*{corollary*}{Corollary}

\begin{document}


\title[Characterisation of $C^\ast$-algebras]{A characterisation of $C^\ast$-algebras through positivity of functionals}

\author{Marcel de Jeu}
\address{Marcel de Jeu, Mathematical Institute, Leiden University, P.O.\ Box 9512, 2300 RA Leiden, The Netherlands}
\email{mdejeu@math.leidenuniv.nl}

\author{Jun Tomiyama}
\address{Jun Tomiyama, Department of Mathematics, Tokyo Metropolitan University, Minami-Osawa, Hachioji City, Japan}
\email{juntomi@med.email.ne.jp}




\begin{abstract}
We show that a unital involutive Banach algebra, with identity of norm one and continuous involution, is a $C^\ast$-algebra, with the given involution and norm, if every continuous linear functional attaining its norm at the identity is positive.
\end{abstract}

\maketitle


If $\alg$ is an involutive Banach algebra, then a linear map $\st:\alg\to\C$ is called positive if $\st(a^\ast a)\geq 0$, for all $a\in \alg$. If the involution is isometric, and $\alg$ has an identity 1 of norm one, then $\st$ is automatically continuous, and $\Vert\st\Vert=\st(1)$, see \cite[Lemma~I.9.9]{Ta}.\footnote{More generally, cf.\ \cite[Theorem~11.31]{Rudin_FA}: even if the involution is not continuous, a positive linear functional is always continuous. If the involution is continuous, and  $\Vert a^\ast\Vert\leq\beta\Vert a\Vert$, for all $a\in\alg$, then $\Vert\st\Vert\leq\sqrt{\beta}\,\st(1)$.} For a unital $C^\ast$-algebra $\alg$, there is a converse: if $\st:\alg\to\C$ is continuous, and $\st(1) = \Vert\st\Vert$, then $\st$ is positive (cf.\ \cite[Lemma~III.3.2]{Ta}). Thus the positive continuous linear functionals on a unital $C^\ast$-algebra are precisely the continuous linear functionals attaining their norm at the identity. Consequently, any Hahn-Banach extension of a positive linear functional, defined on a unital $C^\ast$-subalgebra, is automatically positive again. As is well known, this is a basic characteristic of $C^\ast$-algebras that makes the theory of states on such algebras a success.

If $\alg$ is a unital involutive Banach algebra with identity of norm one, but not a $C^\ast$-algebra, then this converse, as valid for unital $C^\ast$-algebras, need not hold: even when the involution is isometric, there can exist continuous linear functionals on $\alg$ that attain their norm at the identity, but which fail to be positive. For example, for $H^\infty(\D)$, the algebra of bounded holomorphic functions on the open unit disk, supplied with the supremum norm and involution $f^\ast(z)=\overline{f(\bar z)}$ ($z\in\D,\,f\in H^\infty(\D)$), all point evaluations attain their norm at the identity, but only the evaluation in points in $(-1,1)$ are positive. As another example, consider $\loneZ$, the group algebra of the integers. Then its dual can be identified with $\linfZ$, and the continuous linear functionals attaining their norm at the identity are then the bounded maps $\st:\Z\to\C$, such that $\st(0)=\Vert\st\Vert_\infty$. Not all such continuous linear functionals are positive.\footnote{Of course, Bochner's theorem describes the ones that \emph{are} positive.}
 For example, if $\lambda\in\C,\,|\lambda|\leq 1$, then $\st_\lambda\in\linfZ$, defined by $\st(0)=1$, $\st(1)=\lambda$, and $\st(n)=0$ if $n\neq 0,1$, attains its norm at the identity of $\loneZ$. However, if we define $\ell_0:\Z\to\C$ by $\ell_0(0)=1$, $\ell_0(1)=1$, and $\ell_0(n)=0$ if $n\neq 0,1$, then $\ell_0\in\loneZ$, but $\st_0(\ell_0^\ast\ell_0)=2+\lambda$ need not even be real.

It is the aim of this note to show that the existence of examples as above is no coincidence: there necessarily exist continuous linear functionals that attain their norm at the identity, yet are not positive, \emph{because} the algebra in question has a continuous involution, but is not a $C^\ast$-algebra. This is the main content of the result below which, with a rather elementary proof, follows from the far less elementary Vidav-Palmer theorem \cite[Theorem~38.14]{BD}. We formulate the latter first for convenience.

\begin{theorem*}[Vidav-Palmer]
Let $\alg$ be a unital Banach algebra with identity of norm one. Let $\HerA$ be the real linear subspace of all $a\in A$ such that $\st(a)$ is real, for every continuous linear functional $\st$ on $\alg$ such that $\Vert \omega\Vert=\omega(1)$. If $\alg=\HerA + i\, \HerA$, then this is automatically a direct sum of real linear subspaces, and the well defined map $(a_1+ia_2)\mapsto(a_1-ia_2)$ \textup{(}$a_1,a_2\in\HerA$\textup{)} is an involution on $\alg$ which, together with the given norm, makes $\alg$ into a $C^\ast$-algebra.
\end{theorem*}

As further preparation let us note that, if $\alg$ is a unital involutive Banach algebra with identity of norm one and a continuous involution, and if $a\in \alg$ is self-adjoint with spectral radius less than 1, then there exists a self-adjoint element $b\in \alg$ such that $1-a=b^2$. Indeed, using the continuity of the involution, the proof as usually given for a unital Banach algebra with isometric involution, cf.\ \cite[Lemma~I.9.8]{Ta}, which is based on the fact that the coefficients of the power series around 0 of the principal branch of $\sqrt{1-z}$ on $\D$ are all real, goes through unchanged.

\begin{theorem*}
Let $\alg$ be a unital involutive Banach algebra with identity 1 of norm one. Then the following are equivalent:
\begin{enumerate}
\item The involution is continuous, and, if $\st$ is a continuous linear functional on $\alg$ such that $\Vert \st\Vert=\st(1)$, then $\st$ is positive;
\item The involution is continuous, and, if $\st$ is a continuous linear functional on $\alg$ such that $\Vert \st\Vert=\st(1)$, and $a\in \alg$ is self-adjoint, then $\st(a^2)$ is real;
\item $\alg$ is a $C^\ast$-algebra with the given norm and involution.
\end{enumerate}
\end{theorem*}

\begin{proof}
We need only prove that (2) implies (3). Suppose that $a \in \alg$ is self-adjoint and that $\Vert a\Vert<1$. Then, as remarked preceding the theorem, there exists a self-adjoint $b\in \alg$ such that $1-a=b^2$. If $\st$ is a continuous linear functional on $\alg$ such that $\Vert\st\Vert=\st(1)$, then the assumption in (2) implies that $1-\st(a)=\st(1-a)=\st(b^2)$ is real. Hence $\st(a)$ is real. This implies that $\st(a)$ is real, for all self-adjoint $a\in \alg$, and for all continuous linear functionals $\st$ on $\alg$ such that $\Vert\st\Vert=\st(1)$. Since certainly every element of $\alg$ can be written as $a_1+ ia_2$, for self-adjoint $a_1,a_2\in \alg$, this shows that $\alg=\HerA + i\, \HerA$. Then the Vidav-Palmer theorem yields that the involution in that theorem, which agrees with the given one, together with the given norm, makes $\alg$ into a $C^\ast$-algebra.
\end{proof}

In \cite[Theorem~11.2.5]{P2}, a number of equivalent criteria are given for a unital involutive Banach algebra---with a possibly discontinuous involution---to be a $C^\ast$-algebra, but positivity of certain continuous linear functionals is not among them. The proof above of such a criterion is made possible by the extra condition of the continuity of the involution. Although, given the Vidav-Palmer theorem, the proof is quite straightforward, we are not aware of a reference for this characterisation of $C^\ast$-algebras through positivity of linear functionals. Since the result seems to have a certain appeal, we thought it worthwhile to make it explicit.

\end{document}